\documentclass[12pt]{article}
\usepackage{amssymb}
\usepackage{amsmath}
\usepackage{color}
\usepackage{geometry}

\usepackage{color}
\usepackage{amsthm}
\usepackage{amstext}
\usepackage{amssymb}

\newtheorem{thm}{Theorem}
 \newtheorem{lem}{Lemma}
  \newtheorem{rem}{Remark}
  \newtheorem{Corollary}{Corollary}

\title{ {\normalsize\tt\hfill\jobname.tex}\\
Yet again on iteration improvement for averaged expected cost control for 1D ergodic diffusions\\~\\
S.V.~Anulova\footnote{Institute for Information Transmission Problems, Moscow, Russia; email: anulova @ mail.ru} \footnote{For the first author this research has been supported by the 
Russian Foundation for Basic Research 
grant no. \mbox{17-01-00633$\_$a}}, H.~Mai\footnote{CREST and ENSAE ParisTech, France; email: hilmar.mai @ gmail.com} \footnote{The second author thanks the Institut Louis Bachelier for financial support.}, 
A.Yu.~Veretennikov\footnote{University of Leeds, UK, \& National Research University Higher School of Economics, \& Institute for Information Transmission Problems, Moscow, Russia; email: a.veretennikov @ leeds.ac.uk}  \footnote{The third author 
is grateful to the financial support by the DFG through the CRC 1283 ``Taming uncertainty and profiting from randomness
and low regularity in analysis, stochastics and their applications'' at
Bielefeld University during his stay there in August 2017; also, for this author
this study has been funded by the Russian Academic Excellence Project '5-100' and by the 
Russian Foundation for Basic Research grant no. \mbox{17-01-00633$\_$a}. All the authors gratefully acknowledge the support and hospitality of the Oberwolfach Research Institute for Mathematics (MFO) during the RiP programme in June 2014 where this study was initiated. }
}

\begin{document}

\maketitle




\begin{abstract}
 An ergodic Bellman's (HJB) equation is proved for a uniformly ergodic 1D controlled diffusion with variable diffusion and drift coefficients both depending on control; convergence of the values provided by Howard's reward improvement algorithm to the value which is a component of the unique solution of Bellman's equation is established. 
\end{abstract}



\section{\bf Introduction}
The paper is a complete version of the short presentation without detailed proofs in \cite{amv17}. Issues of reliability which was in the title of \cite{amv17}  are not addressed here, all proofs are completed and the results are extended in comparison to the cited article. However, an application to reliability seems  fruitful and is one of the motivations for the present paper; a corresponding remark about it can be found below.
One more motivation is to allow the diffusion coefficient to depend on control. Indirectly, the main result below may be considered as a version of a rigorous realisation of the rather instructive and deliberately non-rigorous example from \cite[Ch. 1, \S 1]{Krylov77} where the point was the vanishing at infinity of the expectation of a current cost. Beside a more detailed calculus in step 3 of the proof, here we tackle  the issue of the HJB equation(s) satisfied everywhere and/or almost everywhere  more precisely 
than in \cite{amv17}.

We consider a one-dimensional stochastic differential equation (SDE) on the probability space $(\Omega,\mathcal{F},(\mathcal{F}_t),P)$ with a one-dimensional $(\mathcal{F}_t)$ Wiener process $B=(B_{t})_{t\geq0}$ with coefficients $b$ and $\sigma$, and with  a stationary control function $\alpha$ (called strategy in the sequel)
\begin{align}\label{sde}
dX^\alpha_{t} & =b(\alpha(X^\alpha_t),X^\alpha_{t})\, dt+\sigma(\alpha(X^\alpha_t), X^\alpha_t)\, dW_{t}, \quad t\ge 0, \nonumber \\ \\  \nonumber 
X^\alpha_{0} & =x. \nonumber
\end{align}

Let a  compact
set  $U\subset\mathbb{R}$ 
be a set where any strategy takes its values. 
The functions $b$ and $\sigma$ on   $U\times\mathbb{R}$ are assumed Borel; later on some further conditions will be imposed, but we note straight away that $\sigma$ will be assumed non-degenerate and that a weak solution of the equation (\ref{sde}) always exists and is Markov and strong Markov, see \cite{Krylov69, Krylov72, Krylov73}. Denote the class of all Borel functions $\alpha$ with values in $U$ by $\cal A$. 
For $u\in U$  and $\alpha(\cdot) \in {\cal A}$ denote 
\[
L^u (x) = 
b(u,x)\frac{d}{dx}\, +\frac12\sigma^2(u,x)\frac{d^2}{dx^2}, 
\quad x\in \mathbb{R}, 
\]
and
\[
L^\alpha (x) = 
b(\alpha(x),x)\frac{d}{dx}\, +\frac12\sigma^2(\alpha(x),x)\frac{d^2}{d x^2},\quad x\in \mathbb{R}.
\]

Denote by ${\cal K}$ the class of functions on $U\times \mathbb R$ (also just on $\mathbb R$) growing no faster than some polynomial. The {\em running cost} function $f$ will always be chosen from this class. The {\em averaged cost} function corresponding to the strategy $\alpha \in {\cal A}$ is then defined as 
\begin{equation}\label{eq0}
\rho^{\alpha}(x):=\limsup_{T\to\infty} \, \frac1{T}\,
\int_{0}^{T}\mathbb{E}_{x}f({\alpha(X_{t}^{\alpha})},X_{t}^{\alpha})\, dt.
\end{equation}
For a strategy $\alpha \in {\cal A}$ the function $f^\alpha:\mathbb{R}\to \mathbb{R},\; f^\alpha(x)=f(\alpha (x), x), \; x\in\mathbb{R}$, is defined. Then  (\ref{eq0}) has an equivalent form
\begin{equation}\label{eq01}
\rho^{\alpha}(x)=\limsup_{T\to\infty} \, \frac1{T}\,
\int_{0}^{T}\mathbb{E}_{x}f^{\alpha}(X_{t}^{\alpha})\, dt.
\end{equation}
Now, the {\em cost function} for the model under consideration  is defined as 
\begin{equation}\label{eq1}
\rho(x):= \inf_{\alpha\in {\cal A}} \limsup_{T\to\infty} \, \frac1{T}\,
\int_{0}^{T}\mathbb{E}_{x}f^{\alpha}(X_{t}^{\alpha})\, dt.
\end{equation}
It will be assumed that for every $\alpha\in {\cal A}$ the solution of the equation (\ref{sde}) $X^\alpha$ is Markov ergodic, i.e., there exists a limiting in total variation distribution $\mu^{\alpha}$ of $X_t^\alpha,\; t\to \infty$, this distribution $\mu^{\alpha}$ does not depend on the initial condition $X_0 = x\in \mathbb R$, is unique and is invariant for the generator $L^\alpha$. The cost function $\rho^{\alpha}$ then does not depend on $x$ and can be rewritten as 
\begin{equation}\label{eq00}
\rho^{\alpha}(x)\equiv \rho^{\alpha} := \int f^{\alpha}(x_{}^{})\,\mu^{\alpha}(dx) =:
\langle f^{\alpha},\mu^{\alpha}\rangle.
\end{equation}
Then what we want to find (compute) is the value
\begin{equation}\label{eq2}
\rho:= \inf_{\alpha\in {\cal A}}
\int f^{\alpha}(x_{}^{})\,\mu^{\alpha}(dx)= \inf_{\alpha\in {\cal A}}
 \langle f^{\alpha},\mu^{\alpha}\rangle.
\end{equation}
For any strategy $\alpha \in {\cal A}$ let us also define an auxiliary function 
\[
v^{\alpha}(x):= \int_0^\infty \mathbb E_x  (f^{\alpha}(X^{\alpha}_{t}) - \rho^\alpha)\, dt.
\]
The convergence of this integral will follow from the assumptions. 

~

{\em The first goal} of this paper is to show  
the {\em ergodic HJB or Bellman's equation} on the pair $(V,\rho)$
\begin{equation}\label{ehjb}
\inf_{u\in U} [L^u V(x) + f^u(x) - \rho]=0, \quad x\in \mathbb R. 
\end{equation}
This assumes showing uniqueness of the second component ($\rho$) along with the property that it coincides with the cost from (\ref{eq2}). The meaning of the first component $V$ will be explained later.
The uniqueness of $V$ will be shown up to an additive constant.
 
The class where the solution $(V,\rho)$ will be studied is the family of all Borel functions $V$ and constants $\rho \in \mathbb R$ such that $V$ has two Sobolev derivatives which are all locally integrable in any power, and $V$ itself should have a moderate grow at infinity not faster than some polynomial. Respectively, the equation (\ref{ehjb}) is to be understood almost everywhere; yet, in the 1D situation and under our assumptions it will follow straightforwardly that this equation is actually  satisfied for all $x\in \mathbb R$. Note that the first derivative can be considered as continuous (due to the embedding theorems), and the second derivative will be always taken Borel, as one of the Borel representatives of Lebesgue's measurable function.

~

{\em The second goal} of the paper is to show how to approach the solution $\rho$ of the main problem by some successive approximation procedure called the Reward   Improvement Algorithm (RIA). 
It is interesting that under our minimal assumptions on regularity of strategies for the weak SDE solution setting it is yet possible to justify a {\em monotonic convergence}  
of the ``exact'' RIA; 
compare to \cite[ch.1, \S 4]{Krylov77} where it was  necessary 
to work with ``approximate'' RIA (called Bellman--Howard's iteration procedure there) and with regularized Lipschitz strategies.

Concerning the equation (\ref{ehjb}), it may look like it lacks some boundary conditions: indeed, a 2nd order PDE normally does require certain boundary conditions, which, for example, in the considered 1D case simply means two boundary conditions at two end-points if the equation is on a bounded interval. However, this is the equation ``in the whole space'' and we are going to solve it in a specific class of functions $V$ -- namely, bounded (if $f$ is assumed bounded), or, at most, moderately growing (if $f$ may admit some moderate growth),  -- which in some sense substitutes the (Dirichlet) boundary conditions at $\pm \infty$. Note that a similar situation can be found in the theory of  Poisson equations in the whole space (see, for example, \cite{PV, Ver2011}).

Concerning a full uniqueness for the solution of (\ref{ehjb}), note that with any solution $(V, \rho)$ and for any constant $C$, the couple $(V + C, \rho)$ is also a solution. There are two close enough options how to tackle this fact: either accept that uniqueness will be established up to a constant, or choose a certain ``natural'' constant satisfying some ``centering condition'' as will be done below.

~

To guarantee ergodicity, we will assume the  ``blanket'' recurrence conditions (see below), which in some sense provide a uniform recurrence for {\em any} strategy. Conditions of this type are sometimes considered too restrictive; however, they do allow  to include  models and cases not covered earlier in this theory and  for this reason we regard this restriction as a  reasonable price for the time being. It is likely that such restrictions may be relaxed so as to include the ``near monotonicity'' type conditions (see \cite{Borkar2012}).

Let us say just a few words about the history of the problem. More can be found in the references provided below.  
Earlier results on ergodic control in continuous time were obtained in \cite{Mandl}, \cite{Morton71}, \cite{Borkar87}, et al. 
In his book \cite{Mandl} Mandl established apparently first results on ergodic (averaged) control for controlled 1D diffusion on a finite interval with boundary conditions  including jumps from the boundary. The author established the HJB equation and proved uniqueness of the couple (up to a constant for the first component). Improvement of control was discussed, too, however, without convergence.

Morton \cite{Morton71} considered the 1D case (a multi-dimensional case too but  under stronger assumptions: we do not touch it in this paper) with a price function defined by (\ref{eq2}) without any relation to (\ref{eq1}). He proved (\cite[Theorem 1]{Morton71}) that the optimal price does satisfy the ergodic Bellman's equation; that the policy determined by Argmax (in our setting Argmin) in the Bellman's equation is optimal within some rather special class of Markov policies which are fixed functions outside some bounded interval; a certain inequality for the optimal price and any solution of Bellman's equation; a remark about RIA; however, neither is the uniqueness for the Bellman's equation solutions established, nor is the convergence of RIA towards a solution proved. 

Discrete time controlled models were considered in the monographs \cite{DynkinYushkevich}, \cite{Howard60},  \cite{Howard71},  \cite{Puterman},  and others, and in the papers \cite{ari1997},  \cite{meyn1997},  \cite{Rykov17}, etc. 

Continuous time controlled processes were treated in the 80s in a chapter of the monograph \cite{Borkar87} where ergodic control for stable diffusions was considered. Arapostathis and Borkar \cite{ari2011}, 
Arapostathis  \cite{ari2012}, Arapostathis,  Borkar and Ghosh \cite{Borkar2012} treated diffusions with ``relaxed control'' and the  diffusion coefficient not depending on the control, under weaker recurrence assumptions (i.e., under two types of condition, stable or near-monotone). In this setting, they establish  Bellman's equation, existence, uniqueness, and RIA convergence. In this paper  we allow the diffusion coefficient to depend on  control and we do not use relaxed control. 

The latest works include  \cite{ari2012}, \cite{Borkar2012}, \cite{Rykov17}, see also the references therein. Although devoted to another type of models -- piecewise-linear Markov ones -- the monograph \cite{CostaDufour} may also be mentioned here.  
In the very first papers and books compact cases with  some auxiliary boundary conditions -- so as to simplify ergodicity -- were studied; convergence of the improvement control algorithms were studied only partially. In later investigations noncompact spaces are allowed; however, apparently, {\em ergodic} control in the diffusion coefficient $\sigma$ of the process has not been tackled earlier. The reader may consult  \cite{Borkar87} and \cite{Krylov77} for research on controlled diffusion processes on a finite horizon, or on infinite horizon with discount (technically equivalent to killing). 

In most of the works on the topic, measurability of the optimal or improved strategy (see below) is assumed. Yet, it is a subtle issue and in our case we give  references -- the basic one is \cite{Stschegolkow} -- and verify the conditions which provide this measurability. 

The paper consists of four sections: 1 -- Introduction,  2 -- Assumptions and some auxiliaries, 3 -- Main result and its proof, and the last one is the Appendix (not numbered).
We will use the convention that arbitrary constants $C$ in the calculus may change from line to line. 

\section{Assumptions and some auxiliaries}\label{sec:aa}

To ensure ergodicity of $X^\alpha$ under any stationary control strategy $\alpha\in {\cal A}$, we make the following assumptions on the drift and diffusion coefficients. 
\begin{enumerate}
\item[(A1)] (boundedness, non-degeneracy, regularity) The functions $b$ and $\sigma$ are Borel bounded in their variables; 
$|b(u,x)|\le C_b$, 
$|\sigma(u,x)|\le C_\sigma$,  $\sigma$ is uniformly  non-degenerate, $|\sigma(u,x)|^{-1}\le C_\sigma$; 
the functions $\sigma(u,x)$, $b(u,x)$, $f^u(x)$ are continuous in $u$ for every $x$. 

\item[(A2)] (recurrence)
\begin{equation}
\lim_{|x|\to \infty}\sup_{u\in U}\, x \,b(u,x)= -\infty.
\label{eq:stability}
\end{equation}

\item[(A3)]  (running cost) The function \(f\)  belongs to the class $\cal K$ of functions 
which 
are Borel measurable in $x$ for each $u$ and admit a uniform in $u$ polynomial bound: there exist constants $C_1, m_1>0$ such that for any $x$, 
\[
\sup_{u\in U} |f^u(x)| \le C_1(1+|x|^{m_1}).
\]

\item[(A4)] (compactness of $U$) The set $U$ is compact.

\medskip

\item[(A5)] (additional regularity)
The functions $b$, $\sigma$, and $f$ are of the class $C^1$ in $x$  for each $u$ with uniformly bounded derivatives.

\end{enumerate}

~

\noindent
We will need the following three lemmata. 
\begin{lem}\label{lem2}
Let the assumptions (A1) -- (A3) hold true. Then 
\begin{itemize}
\item
For any $C_1, m_1 > 0$ there exist $C, m>0$ such that for any strategy $\alpha \in {\cal A}$ and for any function $g$ growing no faster than $C_1 (1+ |x|^{m_1})$, 
\begin{equation}\label{el2}
\sup_{t\ge 0} |\mathbb E_x g(X^\alpha_t)| \le C(1+|x|^m).
\end{equation}

\item
For any $\alpha\in {\cal A}$, the invariant measure $\mu^\alpha$ integrates any polynomial and 
$$
\sup_{\alpha \in {\cal A}}\int |x|^k \, \mu^\alpha(dx) < \infty, \quad \forall \; k>0.
$$

\item
For any strategy $\alpha \in {\cal A}$ the function $\rho^\alpha$ is a constant, and 
\begin{equation}\label{rhoa}
 \sup_{\alpha\in {\cal A}} |\rho^\alpha| \le C <\infty;
\end{equation}
moreover, for any $k>0$ and $f \in {\cal K}$, 
there exist $C,m>0$ such that 
\begin{equation}\label{rhoc}
\sup_{\alpha\in {\cal A}} 
\left|\mathbb{E}_{x}f^{\alpha}(X_{t}^{\alpha})- \rho^\alpha \right| 
\le C\frac{1+|x|^m}{1+t^k},
\end{equation}
and 
\begin{equation}
 \sup_{\alpha\in {\cal A}}\left|\frac1{T}\,
\int_{0}^{T}\mathbb{E}_{x}f^{\alpha}(X_{t}^{\alpha})\, dt - \rho^\alpha \right| \to 0, \quad T\to\infty. 
\end{equation}

\end{itemize}
\end{lem}
{\em Proof}. Follows from \cite[Theorems 5, 6]{Ver99}. Note that in \cite{Ver99} 
the solution of the SDE under investigation should be  weakly unique, and it also must be a homogeneous Markov and strong Markov process; for the equation (\ref{sde}) it is all true by virtue of \cite[Theorem 3]{Krylov69}, \cite{Krylov72},  and \cite[Theorems 2, 3]{Krylov73}, as no continuity of the diffusion coefficient is required for this in the 1D case. (NB: In \cite[Theorem 3]{Krylov73} no continuity is needed even for $D\ge 1$, but then weak uniqueness is established in the 1D case only \cite[Theorem 3]{Krylov69}.)

\begin{Corollary} 
Under the same assumptions,

\begin{equation}\label{bcheb}
\sup_{t\ge 0} |\mathbb E_x 1(|X^\alpha_t| > N)| \le \sup_{t\ge 0} \mathbb E_x \frac{|X^\alpha_t|^m}{N^m} 
\le \frac{C(1+|x|^m)}{N^m}.
\end{equation}
\end{Corollary}
The proof is straightforward by Bienaym\'e -- Chebyshev --Markov's inequality.

\begin{rem}
Note that because $D=1$, under the assumptions (A1)--(A2) for any Borel function $\alpha \in {\cal A}$ there is a unique stationary measure $\mu^\alpha$, which is {\em equivalent to the Lebesgue  measure $\Lambda$}.
The latter follows from the formula for the unique stationary density 
\begin{equation}\label{explicit}
p^\alpha(x) := \frac{d\mu^\alpha(x)}{dx} = 
C_\alpha\frac1{\sigma^2(\alpha(x),x)} \exp\left(2\int_0^x \frac{b(\alpha(y),y)}{\sigma^2(\alpha(y),y)}\,dy\right), 
\end{equation}
where $C_\alpha$ is a normed constant. The fact that $p^\alpha$ is a stationary density can be seen from a substitution to the equation of stationarity $(L^\alpha)^*p=0$ (see, for example, \cite[Lemma 4.16, equation (4.70)]{Khasminskii}); its uniqueness in the class of integrable functions satisfying the normalizing condition  $\int p\,dx = 1$ can be justified via the explicit solution of the stationarity equation in the 1D case which we leave to the readers. 

\end{rem}

In the next Lemma (as well as later in the main Theorem) we use Sobolev spaces $W^{2}_{p, loc}$ with $p>1$. (this notation are taken from \cite[Chapter 2]{lad}, although, in some other sources it is denoted by $W^{2,p}_{loc}$.) Although all main statements can be stated without them, this is done in order to mimick the steps in the proof where these spaces show up naturally due to the direct references, even though the dimension equals one, in which case, of course, some calculus can be simlipified.
\begin{lem}\label{lem1}
Let the assumptions (A1) -- (A3) be satisfied.
Then for any strategy $\alpha \in {\cal A}$ the cost function $v^{\alpha}$ has the following properties:

\begin{enumerate}

\item The function $v^\alpha$ is continuous as well as $(v^\alpha)'$,  and there exist $C,m>0$ both depending only on the constants in (A1)--(A3) such that 
\begin{equation}\label{vineq}
\sup_\alpha (|v^\alpha(x)|  + |v^\alpha(x)'|) \le C (1+|x|^m).
\end{equation}

\item
$v^{\alpha} \in W^{2}_{p, loc}$ for any $p\ge 1$.

\item
$v^{\alpha} \in C^{1,  Lip}$ (i.e., $(v^{\alpha})'$ is locally Lipschitz). 
\item $v^{\alpha}$ satisfies a Poisson equation in the whole space,
\begin{equation}\label{eqva}
L^{\alpha}v^{\alpha}+f^{\alpha} - \langle f^\alpha, \mu^\alpha \rangle \stackrel{}{=} 0,
\end{equation}
in the Sobolev sense; in particular, for almost every $x \in \mathbb R$
\begin{equation}\label{eqvaa}
L^{\alpha}(x)v^{\alpha}(x)+f^{\alpha}(x) - \langle f^\alpha, \mu^\alpha \rangle \stackrel{}{=} 0.
\end{equation}

\item
The solution of the equation (\ref{eqva}) is unique up to an additive constant in the class of Sobolev solutions $W^{2}_{p, loc}$ with any $p>1$ with no more than some (any) polynomial growth of the solution $v^\alpha$.

\item $\langle v^{\alpha},\mu^{\alpha}\rangle =0.$\end{enumerate}
\end{lem}

\begin{proof}
Firstly, the inequality 
\[
\sup_\alpha |v^\alpha(x)| \le C (1+|x|^m)
\]
follows immediately from (\ref{el2}) and from the assumptions. 

Further, let us use a random change of time in the definition of $v^\alpha$:
\begin{align}\label{newv}
v^{\alpha}(x) = \int_0^\infty \mathbb E_x  (f^{\alpha}(X^{\alpha}_{t}) - \rho^\alpha)\, dt
=  \int_0^\infty \mathbb E_x \bar f^\alpha(\bar X^\alpha_s)\,ds, 
\end{align}
where 
\[
\bar f^\alpha(x)= \frac{f^\alpha(x) - \rho^\alpha}{a^\alpha(x)}, 
\]
and $\bar X^\alpha_s$ is the process $X^\alpha_t$ with a changed time which makes the diffusion coefficient equal to one: 
\[
\bar X^\alpha_t := X^\alpha_{t'(t)}, 
\]
where the function \(t'(t)\) is the inverse to the mapping 
\[
t \mapsto \int_0^t \sigma^{2}(X^\alpha_s)\,ds, 
\]
see \cite[Chapter 2.5]{McKean}, or \cite[Theorem 15.5]{GikhmanSkorokhod}. The process $\bar X^\alpha_t$ satisfies an SDE
\begin{equation}\label{newbb}
d\bar X^\alpha_t = d\bar W_t + \bar b^\alpha(\bar X^\alpha_t)dt, \quad 
\bar b^\alpha(x) = \frac{b^\alpha(x)}{\sigma^2(\alpha(x), x)},
\end{equation}
with a new Wiener process $\displaystyle \bar W_t = \int_0^{t'(t)}\sigma(\alpha(X^\alpha_s), X^\alpha_s)\, dW_{s}$, see the same references \cite[Chapter 2.5]{McKean}, or \cite[Theorem 15.5]{GikhmanSkorokhod}.

Further, it follows from (\ref{newv}) and (\ref{newbb})  that the function $v^\alpha$ is a solution of the equation
\begin{equation}\label{newvhat}
\bar L^\alpha v_{}(x) + \bar f^\alpha(x)=0, 
\end{equation}
where
\[
\bar L^\alpha (x)=\bar b(\alpha(x),x)\frac{d}{dx}\, +\frac12\frac{d^2}{d x^2},\quad x\in \mathbb{R}.
\]
Moreover, the last integral in (\ref{newv}) can only converge if $\langle \bar f^\alpha, \bar\mu^\alpha \rangle=0$, where $\bar\mu^\alpha $ is the unique invariant measure of the Markov diffusion $\hat X^\alpha_t$, since otherwise the integral in the right hand side of (\ref{newv}) diverges.  Existence and uniqueness of such an invariant measure (along with a convergence rate) follows, for example, from \cite[Theorem 5]{Ver99} (among many other possible  references) due to the assumption (A1).
The property $v^{\alpha} \in W^{2}_{p, loc}$ for any $p\ge 1$ and the bound 
\[
\sup_\alpha |(v^\alpha)'(x)| \le C (1+|x|^m)
\]
for some $m>0$
follow both from \cite[Theorem 1]{PV01} due to the equation (\ref{newvhat}). 

Further, given (\ref{vineq}), the bound \(v^{\alpha} \in C^{1,  Lip}\) (which means a local, not global Lipschitz condition for \((v^{\alpha})'\)) follows from the equation (\ref{newvhat}), as \((v^{\alpha})''\) turns out to be locally bounded by virtue of this equation. The same equation(\ref{newvhat})  implies (\ref{eqva}) and (\ref{eqvaa}). Uniqueness of solution for the equation (\ref{newvhat}) and, hence, also for (\ref{eqva}) up to an additive constant follows from 
\cite{PV01}; see also \cite[Lemma 4.13 and Remark 4.3]{Khasminskii}. Finally, the last assertion of the Lemma is due to the Fubini theorem, 
\[
\int v^{\alpha}(x)\mu^\alpha(dx) =\int \int_0^\infty \mathbb E_x  (f^{\alpha}(X^{\alpha}_{t}) - \rho^\alpha)\, dt\mu^\alpha(dx) 
= \int_0^\infty \int\mathbb E_x  (f^{\alpha}(X^{\alpha}_{t}) - \rho^\alpha)\mu^\alpha(dx)\, dt=0, 
\]
by virtue of the absolute convergence
\[
\int \int_0^\infty \mathbb |E_x  (f^{\alpha}(X^{\alpha}_{t}) - \rho^\alpha)|\, dt\mu^\alpha(dx) 
<\infty.
\]

~

\end{proof}

\begin{lem}\label{lem3}
Let the assumptions (A1) -- (A2) hold true. Then $\exists \; 0<C_1<C_2$ such that for any strategy $\alpha$ for the constant $C_\alpha$ from (\ref{explicit}) we have, 
\[
C_1 \le C_\alpha \le C_2. 
\]
Also, for any $k$ there is a constant $C$ such that for every $x$ uniformly in $\alpha$
\[
p^\alpha(x) \le \frac{C}{1+|x|^k},  
\]
and there exist constants $c, \kappa >0$ such that uniformly in $\alpha$
\[
p^\alpha(x) \ge c\exp(-\kappa|x|).  
\]

\end{lem}
{\em Proof}. Follows straightforwardly from the recurrence and boundedness assumptions and from the formula (\ref{explicit}).

\section{\bf Main results} 
We
accept in this section that a solution of the SDE with any Markov strategy exists and is a {\em weak} solution. However, it is important in the proof that it is unique in distribution, strong Markov and Markov ergodic; repeat what was already mentioned in the proof of the Lemma \ref{lem2}, that all of these follow from \cite{Krylov69} and from the assumptions (A1) and (A2) (see \cite{Ver99} about ergodicity).

For any pair $(v, \rho): \; v\in \bigcap_{p>1} W^{2}_{p, loc}, \, \rho\in \mathbb{R}$, 
define
\[
 F[v, \rho](x) := \inf_{u\in U} \left[L^{u}v(x)+f^{u}(x) -\rho\right], 
 \quad G[v](x):= \inf_{u \in U} \left[L^{u}v(x)+f^{u}(x) \right],
\]
and
\[  
F_1[v', \rho](x):= \inf_{u\in U} [\hat b^u v' + \hat f^u - \hat \rho](x), 
\]
where
\begin{eqnarray*}
a^u(x) = \frac12 (\sigma^u(x))^2, \quad \hat b^u(x) = b^u(x)/a^u(x), 
 \\ \hat f^u(x) = f^u(x)/a^u(x), \quad
\hat \rho^u(x) = \rho/a^u(x).
\end{eqnarray*}
The functions $v$ and $v'$ may be regarded as continuous and absolutely continuous due to the embedding theorems  \cite{lad}. The function $F[v, \rho](\cdot)$ is defined by the formula above as a function of the class $L_{p, loc}$ for any $p>1$; in particular, it is Lebesgue measurable and as such it is defined only  a.e. with respect to $x$. We may and will use a (any)  Borel measurable version of the function $F[v, \rho]$, the existence of which follows, for example,  from Luzin's Theorem \cite{lebesgue-borel}). It will be shown in the sequel that the function $F_1[v', \rho](x)$ is continuous in $x$ and locally Lipschitz in the two other variables. 

Let us recall what a reward improvement algorithm (RIA) is. We start with some (any) stationary strategy  $\alpha_0 \in {\cal A}$. Denote the corresponding cost, the invariant measure, and the auxiliary function  $\rho_0=\rho^{\alpha_0}=  \langle f^{\alpha_0}, \mu^{\alpha_0}\rangle$, and  $v_0=v^{\alpha_0}$. If for some $n=0,1, \ldots$ the triple $(\alpha_n, \rho_n, v_n)$ is determined,  then the strategy $\alpha_{n+1}$ is defined as follows: for  a.e.  $x$ the function $\alpha_{n+1}$ is chosen so that for each $x$
\begin{align}\label{eqvn}
L^{\alpha_{n+1}}v_n(x)+f^{\alpha_{n+1}}(x)=   G[v_n](x),
\end{align}
or, in other words,  
\[
\alpha_{n+1}(x) \in \mbox{Argmin}_{u \in U}\left[L^{u}v_n(x)+f^{u}(x) \right].
\]
We assume that a Borel measurable version of such strategy may be chosen; see the reference in the Appendix. To this strategy 
$\alpha_{n+1}$ there correspond the unique invariant measure \(\mu^{\alpha_{n+1}}\), the value \(
\rho_{n+1}:= \langle f^{\alpha_{n+1}}, \mu^{\alpha_{n+1}}\rangle
\), and the function $v_{n+1}= v^{\alpha_{n+1}}$. 

\begin{thm}\label{thm2}
Let the assumptions (A1) -- (A4) be satisfied. 
Then:

1. For any $n$, $\rho_{n+1}\le \rho_n$, and there exists a limit
\(
\rho_n \downarrow \tilde \rho.
\)

2. The sequence $(v_n)$ is tight in $C^1[-N,N]$ for each $N>0$, and 
there exists a bounded sequence of constants $\beta_n$ such that there exists a limit $\lim_{n} (v_n(x) + \beta_n) =: \tilde v(x)$.

3. The couple \((\tilde v, \tilde \rho)\) solves the equation (\ref{ehjb}). 

4. This solution \((\tilde v, \tilde \rho)\)  is unique 
-- up to an additive constant for $\tilde v$ --  in the class of functions growing no faster than some (any) polynomial and belonging to the class $W_{p, loc}^{2}$ for any $p>0$ for the first component and for $\tilde \rho \in \mathbb R$.

5. The component $\tilde \rho$ in the couple $(\tilde v, \tilde \rho)$  coincides with $\rho$.

6. Under the additional assumption (A5), \(\tilde v'' \in Lip_{loc}\). 

\end{thm}
In the short presentation \cite{amv17}, beside the restrictive assumption $f\in [0,1]$ and maximisation instead of minimisation, only a sketch of the proof was offered with many details explained too  briefly; uniqueness of $\tilde v$ was not addressed. Here the full proof is given. NB: 
We never  compare the trajectories of two SDE solutions in one formula and the processes corresponding to different strategies may be defined on different probability spaces. 

~

\begin{proof}
{\bf 1}. Due to  (\ref{eqvn}) and (\ref{eqva}), for almost every (a.e.) $x\in \mathbb R$, 
\begin{eqnarray*}
\rho_{n} 
= L^{\alpha_{n}}v_n(x)+f^{\alpha_{n}}(x)
\ge G[v_n](x) 
= L^{\alpha_{n+1}}v_n(x)+f^{\alpha_{n+1}}(x)
\end{eqnarray*}
and also for a.e. $x\in \mathbb R$, 
\[
\rho_{n+1} =  L^{\alpha_{n+1}}v_{n+1}(x)+f^{\alpha_{n+1}}(x)
\]
So,
\begin{align}\label{erhorho}
\rho_n - \rho_{n+1} \stackrel{a.e.}{\ge} (L^{\alpha_{n+1}}v_{n} + f^{\alpha_{n+1}})(x) - (L^{\alpha_{n+1}}v_{n+1} + f^{\alpha_{n+1}})(x)
 \nonumber\\\\ \nonumber
= (L^{\alpha_{n+1}}v_{n}  - L^{\alpha_{n+1}}v_{n+1})(x).
\end{align}
Let us apply Ito -- Krylov's  formula (see \cite{Krylov77}) with expectations (also known as Dynkin's formula) to \((v_{n}  - v_{n+1})(X^{\alpha_{n+1}}_t)\):   
we have for any $x\in \mathbb R$, 
\begin{eqnarray}\label{eexpexp}
\mathbb E_x \left(v_{n}(X^{\alpha_{n+1}}_t) - v_{n+1}(X^{\alpha_{n+1}}_t)\right) -
\left(v_{n} - v_{n+1}\right)(x)
 \nonumber \\
 \\\nonumber 
= \mathbb E_x \int_0^t (L^{\alpha_{n+1}}v_{n}  - L^{\alpha_{n+1}}v_{n+1})(X^{\alpha_{n+1}}_s)\,ds
\le  \mathbb E_x \int_0^t (\rho_n - \rho_{n+1})\,ds = (\rho_n - \rho_{n+1})\,t.
\end{eqnarray}
The equality in the equation (\ref{eexpexp}) holds for all  $x\in\mathbb R$ and not just a.e. since the functions $v_n$ are Sobolev solutions of Poisson equations locally integrable in any degree with their derivatives up to the second order. Such functions can be regarded as  continuous due to the embedding theorems \cite{lad}. In addition, the functions 
$\mathbb E_x v_{n}(X^{\alpha_{n+1}}_t)$,  $\mathbb E_x v_{n+1}(X^{\alpha_{n+1}}_t)$, and $\displaystyle  \mathbb E_x \int_0^t (L^{\alpha_{n+1}}v_{n}  - L^{\alpha_{n+1}}v_{n+1})(X^{\alpha_{n+1}}_s)\,ds$ as functions of $x$ for each $t>0$ are all H\"older continuous, being solutions of non-degenerate parabolic equations \cite{KS}. 
We also used the fact that the distribution of $X^{\alpha_{n+1}}_s$ for almost all $s>0$ is absolutely continuous with respect to the Lebesgue measure due to the non-degeneracy and by virtue of Krylov's estimates \cite{Krylov77}; due to this reason and because $v_n, v_{n+1} \in C$, the a.e. inequality (\ref{erhorho}) implies (\ref{eexpexp}) for every $x$. 
Further, since the left hand side in (\ref{eexpexp}) is bounded for a fixed $x$ by virtue of the Lemma \ref{lem1}, we divide all terms of the latter inequality by \(t\) and let \(t \to\infty\)  to get,
\[
0 \le \rho_n - \rho_{n+1},
\]
as required. Thus, $\rho_n \ge \rho_{n+1}$, so that $\rho_n \downarrow \tilde \rho$ (since the sequence $\rho_n$ is bounded for $f \in {\cal K}$, see (\ref{rhoa}) in the Lemma \ref{lem2}) with some $\tilde \rho$.  So,  the RIA does converge. 

\medskip

Note that clearly $\tilde \rho \ge \rho$, since $\rho$ is the infimum over all Markov strategies, while $\tilde \rho$ is the infimum   over some countable subset  of them.  Later on we shall show that they do coincide.
 
\medskip
 
Now we want to show that 
there exists a bounded sequence of real values (non-random!) $\{\beta_n\}$ such that $v_n + \beta_n\to \tilde v$, so that the couple $(\tilde v, \tilde \rho)$ satisfies the equation (\ref{ehjb}),
and that $\tilde \rho$ here is unique, as well as $\tilde v$ in some sense. In the first instance we will do it for some subsequence $n_j$; eventually the convergence of the whole sequence $v_n$ will follow from the uniqueness of the solution of Bellman's equation, although, it is not important for the proof of the Theorem.

~

\noindent
{\bf 2}. Let us show local tightness of the family of functions $(v_n)$ in $C^1$. 
Note that the equation (\ref{ehjb}) is equivalent to the following:
\begin{equation}\label{ehjb2}
V''_{}(x)+ \inf_{u\in U}\left[\frac{b(u,x)}{a(u,x)} V'_{}(x)+\frac{f(u,x)}{a(u,x)} -\frac{\rho_{}}{a(u,x)}\right] = 0, 
\end{equation}
while the equation
\begin{equation}\label{zero}
L^{\alpha_{n+1}}v_{n+1}(x)+f^{\alpha_{n+1}}(x) -\rho_{n+1} 
\stackrel{a.e.}{=} 0 
\end{equation}
is equivalent to 
\[
v''_{n+1}(x)+ \frac{b(\alpha_{n+1}(x),x)}{a(\alpha_{n+1}(x),x)} v'_{n+1}(x)+\frac{f(\alpha_{n+1}(x),(x))}{a(\alpha_{n+1}(x),x)} -\frac{\rho_{n+1}}{a(\alpha_{n+1}(x),x)} = 0.
\]
According to the Lemma \ref{lem1}, the functions $v_{n+1}'$ are uniformly locally bounded. Since the sequence $\rho_{n+1}$ is bounded and due to the uniform local boundedness of the functions $f(\alpha_{n+1}(x),x)$ and uniform nondegeneracy of $a$, it follows that $(v''_n)$ locally are uniformly bounded and satisfy the uniform in $n$ growth bounds similar to (\ref{vineq}) for the function itself and for its first derivative due to the equation (for example, due to (\ref{ehjb2})). This guarantees compactness of $(v_n)$ in $C^1$ locally. 

~

\noindent
{\bf 3}. 
Due to the (local) compactness property showed in the previous step, by the diagonal procedure from any infinite sub-family of functions $v_n$ it is possible to choose a converging in $C^1_{loc}$ subsequence. We want to show that up to a constant the limit is unique. For this aim, first of all we shall see shortly that if some $v_{n_j}(x)$ has a limit as $n_j\to\infty$, say, $\tilde v(x)$ (locally in $C$) then $v_{n_j+1}(x)+\beta_{n_j}$ has the same limit, where  $\beta_{n}$ is some bounded sequence of real values. (In fact, what will be established is a little bit more complicated but still enough for our purposes.) We have, 
\[
L^{\alpha_{n+1}}v_{n+1}(x)+f^{\alpha_{n+1}}(x) -\rho_{n+1} \stackrel{a.e.}{=} 0, 
\] 
and 
\begin{equation}\label{psi}
L^{\alpha_{n+1}}v_{n}(x)+f^{\alpha_{n+1}}(x) -\rho_{n} =: -\psi_{n+1}(x) \stackrel{a.e.}{\le} 0.
\end{equation}
Let us rewrite it as follows, 
\[
L^{\alpha_{n+1}}v_{n}(x)+f^{\alpha_{n+1}}(x) -\rho_{n} + \psi_{n+1}(x) \stackrel{a.e.}{=} 0.
\] 
In other words, the function $v_{n}$ solves the Poisson equation with the second order operator $L^{\alpha_{n+1}}$ and the ``right hand side''  
$-(f^{\alpha_{n+1}}(x) + \psi_{n+1}(x) -\rho_{n})$. This is only possible if the expression $f^{\alpha_{n+1}}(x) + \psi_{n+1}(x) -\rho_{n}$ is centered with respect to the invariant measure  $\mu^{n+1}$ because Poisson equations in the whole space have no solutions for non-centered right hand sides (see, for example, \cite{PV01}). This implies that 
\[
\langle f^{\alpha_{n+1}}(x) + \psi_{n+1} -\rho_{n}, \mu^{n+1}\rangle = 0
\]
So, 
\begin{equation}\label{psi1}
\langle \psi_{n+1}, \mu^{n+1}\rangle = \rho_n - \rho_{n+1}.
\end{equation}
Now denote 
\[
w_n(x) := v_n(x) - v_{n+1}(x). 
\]
We have, 
\[
L^{\alpha_{n+1}}w_{n}(x) + \psi_{n+1}(x) -(\rho_{n}-\rho_{n+1})  \stackrel{a.e.}{=} 0.
\] 
So, there is a constant $\beta_n = \langle w_n, \mu^{n+1}\rangle$ such that 
\begin{equation}\label{wn}
w_{n}(x) - \beta_n = \int_0^\infty \mathbb E_x (\psi_{n+1}(X^{n+1}_t) -(\rho_{n}-\rho_{n+1}))\, dt.
\end{equation}
Let us show that for any $N>0$, 
\begin{equation}\label{psi2}
\int_{-N}^N \psi^2_n(x)\,dx \to 0, \quad n\to \infty. 
\end{equation}
First of all, note that  all functions $\psi_n$ and, hence, $\psi^2_n$ are uniformly locally bounded and may only grow polynomially fast, 
\begin{equation}\label{psibd}
(0\le \;) \; \psi_n(x) \le C(1+|x|^m), 
\end{equation}
with some $C,m$ the same for all values of $n$. 
which follows from the definition (\ref{psi}), and the properties of derivatives $v'_n$ and $v''_n$, and from the Lemma \ref{lem3}, and due to
\[
\langle \psi_{n+1}, \mu^{n+1}\rangle = \rho_n - \rho_{n+1} \to 0, \quad n \to \infty. 
\]
Now let us rewrite the equation (\ref{wn}) 
via a stationary version of our diffusion, say, $\tilde X^{n+1}_t$:
\[
w_{n}(x) - \beta_n = \int_0^\infty \mathbb E_x (\psi_{n+1}(X^{n+1}_t) - E_{\mu^{n+1}} (\psi_{n+1}(\tilde X^{n+1}_t))\, dt.
\] 
(Note that if we knew that $w_n$ were centered with respect to the invariant measure $\mu^{n+1}$ then we would have $\beta_n=0$; however, the functions $v_n$ and $v_{n+1}$ are both centered with respect to two different measures, and this is the reason why their difference is not just small, but small up to some additive constant; this very constant is denoted by $\beta_n$.)
Using the coupling idea (see, for example, \cite{Ver99}), let us consider the independent processes $X^{n+1}_t$ and $\tilde X^{n+1}_t$ on the same probability space (just considering the product space) and denote the moment of the first meeting
\[
\tau: = \inf(t\ge 0: \, X^{n+1}_t = \tilde X^{n+1}_t). 
\]
It is known (see \cite[Theorem 5]{Ver99}) that under our recurrence assumptions for any $k>0$ there are some constants $C_k, m$ such that uniformly with respect to $n$, 
\[
\mathbb E_{x, \mu^{n+1}} \tau^k \le C_k(1+|x|^m).
\]
Denote 
\[
\hat X^{n+1}_t := 1(t<\tau) X^{n+1}_t + 
1(t\ge \tau) \tilde X^{n+1}_t. 
\]
Since $\tau$ is a stopping time and because the couple $(X^{n+1}_t, \tilde X^{n+1}_t)$ is strong Markov (see \cite{Krylov73}), the process
$(\hat X^{n+1}_t)$ is also strong Markov equivalent to $(X^{n+1}_t)$. Therefore, it is possible to rewrite, 
\[
w_{n}(x) - \beta_n = \int_0^\infty \mathbb E_{x, \mu} (\psi_{n+1}(\hat X^{n+1}_t) - \psi_{n+1}(\tilde X^{n+1}_t))\, dt.
\] 
Hence, using the fact that after $\tau$ the processes $\hat X^{n+1}_t$  and $\tilde X^{n+1}_t$ coincide, we obtain
\begin{align*}
w_{n}(x) - \beta_n = \int_0^\infty \mathbb E_{x, \mu} 1(t<\tau)(\psi_{n+1}(\hat X^{n+1}_t) - \psi_{n+1}(\tilde X^{n+1}_t))\, dt
 \\\\
=  \int_0^\infty \mathbb E_{x, \mu} \sum_{i=0}^\infty 1(i\le \tau<i+1) 1(t<\tau)(\psi_{n+1}(\hat X^{n+1}_t) - \psi_{n+1}(\tilde X^{n+1}_t))\, dt
 \\\\
=  \sum_{i=0}^\infty \mathbb E_{x, \mu}  \int_0^\infty 1(i\le \tau<i+1) 1(t<\tau)(\psi_{n+1}(\hat X^{n+1}_t) - \psi_{n+1}(\tilde X^{n+1}_t))\, dt. 
\end{align*}
Thus, using Cauchy-Buniakovsky-Schwarz inequality and Fubini Theorem, we have, 
\begin{align*}
|w_{n}(x) - \beta_n|
\le \sum_{i=0}^\infty \mathbb E_{x, \mu}  \int_0^{i+1} 1(\tau>i)|\psi_{n+1}(\hat X^{n+1}_t) - \psi_{n+1}(\tilde X^{n+1}_t)|\, dt
 \\\\
\le \sum_{i=0}^\infty  \int_0^{i+1} \mathbb E_{x, \mu}  1(\tau>i)(|\psi_{n+1}(\hat X^{n+1}_t)| +|\psi_{n+1}(\tilde X^{n+1}_t)|)\, dt 
 \\\\
\le \sum_{i=0}^\infty  \int_0^{i+1} (\mathbb E_{x, \mu}  1(\tau>i))^{1/2}(\mathbb E_{x, \mu} (|\psi_{n+1}(\hat X^{n+1}_t)| +|\psi_{n+1}(\tilde X^{n+1}_t)|)^2)^{1/2}\, dt 
 \\\\
\le 2 \sum_{i=0}^\infty  (\mathbb E_{x, \mu}  1(\tau>i))^{1/2}\int_0^{i+1} (\mathbb E_{x, \mu} |\psi_{n+1}(\hat X^{n+1}_t)|^2 +\mathbb E_{x, \mu}|\psi_{n+1}(\tilde X^{n+1}_t)|^2)^{1/2}\, dt 
 \\\\
\le 2 \sum_{i=0}^\infty  (\mathbb E_{x, \mu}  1(\tau>i))^{1/2}\int_0^{i+1} [(\mathbb E_{x, \mu} (\psi_{n+1}(\hat X^{n+1}_t))^2)^{1/2} +(\mathbb E_{x, \mu}(\psi_{n+1}(\tilde X^{n+1}_t))^2)^{1/2}]\, dt. 
\end{align*}
Now, let us take any $\epsilon>0$ and use the inequality
\(
\sqrt{a} \le \frac\epsilon2 + \frac{a}{2\epsilon}
\).
We estimate, 
\begin{align*}
\int_0^{i+1} [(\mathbb E_{x, \mu} (\psi_{n+1}(\hat X^{n+1}_t))^2)^{1/2} +(\mathbb E_{x, \mu}(\psi_{n+1}(\tilde X^{n+1}_t))^2)^{1/2}]\, dt
 \\\\
\le  \epsilon (i+1) 
+ \frac1{2\epsilon} \int_0^{i+1} [\mathbb E_{x, \mu} \psi^2_{n+1}(\hat X^{n+1}_t) +\mathbb E_{x, \mu}\psi^2_{n+1}(\tilde X^{n+1}_t)]\, dt.
\end{align*}
Let us first consider the stationary term. We have, 
\begin{align*}
\frac1{2\epsilon} \int_0^{i+1} \mathbb E_{x, \mu}\psi^2_{n+1}(\tilde X^{n+1}_t)\, dt 
+ \frac1{2\epsilon} \int_0^{i+1} \mathbb E_{x, \mu}\psi^2_{n+1}(\hat X^{n+1}_t)\, dt 
 \\\\
=  \frac1{2\epsilon} \int_0^{i+1} \mathbb E_{x, \mu}\psi^2_{n+1}1_{[-N,N]}(\tilde X^{n+1}_t)\, dt 
+ \frac1{2\epsilon} \int_0^{i+1} \mathbb E_{x, \mu}\psi^2_{n+1}1_{\mathbb R \setminus [-N,N]}(\tilde X^{n+1}_t)\, dt
 \\\\
+ \frac1{2\epsilon} \int_0^{i+1} \mathbb E_{x, \mu} \psi^2_{n+1} 1_{[-N,N]}(\hat  X^{n+1}_t)\, dt 
+ \frac1{2\epsilon} \int_0^{i+1} \mathbb E_{x, \mu}\psi^2_{n+1}1_{\mathbb R \setminus [-N,N]}(\hat X^{n+1}_t)\, dt. 
\end{align*}
Given (\ref{psibd}) 
and because any stationary measure integrates uniformly any power function, let us find such $N$ that uniformly with respect to $n$, 
\begin{equation}\label{eN}
\langle C(1+|x|^{2m})1_{\mathbb R \setminus [-N,N]}, \mu^{n+1}
\rangle
< \epsilon^2/2,
\end{equation}
which is possible due to the Lemmata \ref{lem2} and \ref{lem3}, and also such that  
\(
N > \epsilon^{-2}
\).
Then choose $n(\epsilon)$ such that 
\[
\sup_{n\ge n(\epsilon)} \int_{|x|\le N} \psi^2_n(x)\,dx < \epsilon^2/2. 
\]
Due to Krylov's estimate 
$$
\displaystyle \mathbb E \int_0^T g(\tilde X^{n+1}_t)\, dt \le K_T\|g\|_{L_1(\mathbb R)}
$$
for any function $g\ge 0$, and also 
$$
\displaystyle \mathbb E \int_s^{s+T} g(\tilde X^{n+1}_t)\, dt \le K_T\|g\|_{L_1(\mathbb R)}
$$ 
for any $s>0$ (follows from \cite[Theorem 2.2.3]{Krylov77}), 
we evaluate with $n\ge n(\epsilon)$:
\begin{align*}
\frac1{2\epsilon} \int_0^{i+1} \mathbb E_{x, \mu}\psi^2_{n+1}1_{[-N,N]}(\tilde X^{n+1}_t)\, dt
 \\\\
=  \frac{1}{2\epsilon} \sum_{k=0}^{i} \mathbb E_{x} \int_k^{k+1} \psi^2_{n+1}1_{[-N,N]}(\tilde X^{n+1}_t)\, dt 
\le \frac{i+1}{2\epsilon} 
K \|\psi^2_{n+1}1_{[-N,N]}\|_{L^1_{}}
\le \frac{(i+1)K\epsilon}2.
\end{align*}
Indeed, for any $k\ge 0$ we have, 
\begin{align*}
\mathbb E_{x} \int_k^{k+1} \psi^2_{n+1}1_{[-N,N]}(\tilde X^{n+1}_t)\, dt 
= \mathbb E_{x} \mathbb E_{x}\left(\int_k^{k+1} \psi^2_{n+1}1_{[-N,N]}(\tilde X^{n+1}_t)\, dt|{\cal F}_{k}\right) 
 \\\\
= \mathbb E_{x} \mathbb E_{\tilde X^{n+1}_k} \int_0^{1} \psi^2_{n+1}1_{[-N,N]}(\tilde X^{n+1}_t)\, dt 
 \\\\
\le \frac{1}{2\epsilon} 
K \|\psi^2_{n+1}1_{[-N,N]}\|_{L^2_{}}
\le \frac{(i+1)K\epsilon}2.
\end{align*}

This argument works for the non-stationary process as well: 
due to the same Krylov's estimate, 
\begin{align*}
\frac1{2\epsilon} \int_0^{i+1} \mathbb E_{x, \mu}\psi^2_{n+1}1_{[-N,N]}(\hat X^{n+1}_t)\, dt
 \\\\
=  \frac{1}{2\epsilon} \sum_{k=0}^{i} \mathbb E \int_k^{k+1} \psi^2_{n+1}1_{[-N,N]}(\hat X^{n+1}_t)\, dt 
\le \frac{i+1}{2\epsilon} 
K \|\psi^2_{n+1}1_{[-N,N]}\|)_{L^1_{}}
\le \frac{(i+1)K\epsilon}2.
\end{align*}
Further, 
\begin{align*}
\frac1{2\epsilon} \int_0^{i+1} \mathbb E_{x, \mu}\psi^2_{n+1}1_{\mathbb R \setminus [-N,N]}(\tilde X^{n+1}_t)\, dt 
\le \frac{i+1}{2\epsilon}  \times \frac{\epsilon^2}2 = \frac{(i+1) \epsilon}4.
\end{align*}
Finally, using (\ref{bcheb}), we obtain with some $m$,  
\begin{align*}
\frac1{2\epsilon} \int_0^{i+1} \mathbb E_{x, \mu}\psi^2_{n+1}1_{\mathbb R \setminus [-N,N]}(\hat X^{n+1}_t)\, dt 
 = \frac1{2\epsilon} \int_0^{i+1} \mathbb E_{x, \mu}\psi^2_{n+1}1_{\mathbb R \setminus [-N,N]}(X^{n+1}_t)\, dt 
 \\\\
\le  C\frac{i+1}{2\epsilon} \frac{(1+|x|^m)}{N}
\le C (i+1)(1+|x|^m)\epsilon.
\end{align*}
Overall, 
this shows that with the appropriately chosen (uniformly bounded) $\beta_n$, 
\begin{align}\label{wbn}
|w_{n}(x) - \beta_n| \le C (1+|x|^{2m}) \epsilon  \sum_{i=0}^\infty  (i+1)(\mathbb E_{x, \mu}  1(\tau>i))^{1/2},  
\quad n\ge n(\epsilon).
\end{align}
By virtue of the results in \cite{Ver99}, for any $k>0$ there are  $C,m >0$ such that 
\[
\mathbb P_{x, \mu}  1(\tau>i) \le C \frac{1+|x|^m}{1+i^k}. 
\]
Therefore, taking any $k>1$, we have that the series in (\ref{wbn}) converges providing us an estimate 
\begin{align}\label{wbne}
|w_{n}(x) - \beta_n| \le C (1+|x|^{3m}) \epsilon,  
\quad n\ge n(\epsilon).
\end{align}
In other words, the difference $w_{n}(x) - \beta_n = v_n - v_{n+1} - \beta_n$ is locally uniformly converging to zero as $n\to\infty$.
Naturally, it also implies that for any  subsequence $n_j$ such that $v_{n_j}$ converges locally uniformly in $C^1$ we have that $v'_{n_j}$ and $v'_{n_j+1}$ may only converge to the same limit, i.e., 
derivatives $v'_{n_j} - v'_{n_j+1} \to 0$ (locally uniformly) as $j\to\infty$. Indeed, otherwise we just integrate to show that the limits of $v_{n_j}$ and $v_{n_j+1} + \beta_{n_j}$ are different, which contradicts to what was established earlier. 

~

\noindent
{\bf 4}. What we want to do now is to pass to the limit as $j \to \infty$ in the equations 
\[
L^{\alpha_{n_j+1}}v_{n_j+1}(x)+f^{\alpha_{n_j+1}}(x) -\rho_{n_j+1} \stackrel{a.e.}{=} 0, \quad \& \quad 
G[v_{n_j}](x) - \rho_{n_j} \le 0, 
\] 
where $(n_j, j \to \infty)$ is any sequence such that $v_{n_j}$ converges (locally uniformly) in $C^1$.  
From 
\begin{eqnarray*}
G[v_{n_j}](x) - \rho_{n_j} =
L^{\alpha_{n_j+1}}v_{{n_j}}(x)+f^{\alpha_{n_j+1}}(x) -\rho_{n_j} 
 \\\\
(=
\inf_{u\in U} [L^{u}v_{n_j}(x)+f^{u}(x) -\rho_{n_j}] \stackrel{a.e.}{\le} 0),
\end{eqnarray*}
by subtracting zero a.e. (\ref{zero}), we obtain a.e., 
\begin{equation}\label{eG}
G[v_{n_j}](x) - \rho_{n_j} =
L^{\alpha_{n_j+1}}(v_{n_j}(x) - v_{n_j+1}(x))  -(\rho_{n_j}-\rho_{n_j+1}).
\end{equation}
Now we want to show that 
\begin{equation}\label{133}
\tilde v'(x) - \tilde v'(r) + \int_r^x F_1[s, \tilde v'(s), \tilde\rho]\,ds = 0, 
\end{equation}
which in turn implies by differentiation the equation equivalent to (\ref{ehjb}), 
\begin{equation}\label{1333}
\tilde v''(x) + F_1[x,\tilde v', \tilde\rho](x) \stackrel{}{=} 0, 
\end{equation}
for any $x$, with the note that $\tilde v'$ is absolutely continuous. 

Let us show that (\ref{eG}),
indeed, implies (\ref{133}). 
Note that $G[v_{n_j}](x) - \rho_{n_j} \le 0$ (a.e.). 
Let us divide (\ref{eG}) by 
$a_{n_j+1} = a^{\alpha_{n_j+1}}$ and use $\delta:=\inf_{u,x} a^u(x) > 0$: we get a.e. with some $K>0$, 
\begin{align}\label{also}
\displaystyle 0 \ge \frac{(G[v_{n_j}](x) - \rho_n)}{a_{n_j+1}}
= (v''_{n_j}(x) - v''_{n_j+1}(x)) + (\hat b^{\alpha_{n_j+1}} (v'_{n_j} - v'_{n_j+1})) - \frac{(\rho_{n_j}-\rho_{n_j+1})}{a_{n_j+1}}
 \nonumber \\  \nonumber \\
\ge (v''_{n_j}(x) - v''_{n_j+1}(x)) - \frac{K}\delta |v'_{n_j}(x) - v'_{n_j+1}(x)|  - \frac1\delta (\rho_{n_j}-\rho_{n_j+1}).
\end{align}
So, we have just shown that a.e., 
  \begin{eqnarray}\label{already}
\displaystyle 0 \ge (v''_{n_j}(x) - v''_{n_j+1}(x)) - \frac{K}\delta |v'_{n_j}(x) - v'_{n_j+1}(x)|  - \frac{\rho_{n_j}-\rho_{n_j+1}}{\delta}.
\end{eqnarray}
The next trick is to note that again due to (\ref{also}) and $\rho_{n_j}\ge \rho_{n_j+1}$, and since $\delta \le a \le C$, 
\[
0 \stackrel{a.e.}{\ge} G[v_{n_j}] (x)-\rho_{n_j} \ge a_{n_j+1}(v''_{n_j} - v''_{n_j+1})(x) - C'|v'_{n_j}-v'_{n_j+1}|(x)- (\rho_{n_j}-\rho_{n_j+1}),
\]
which implies that with some $C, c>0$, 
\begin{equation}\label{vF}
0\stackrel{a.e.}{\ge} v''_{n_j} + F_1[v'_{n_j}, \rho_{n_j}]  \ge 
((v''_{n_j} - v''_{n_j+1}) - C|v'_{n_j} - v'_{n_j+1}|)  -c(\rho_{n_j}-\rho_{n_j+1}).
\end{equation}
Since $v_{n_j}'$ is absolutely continuous, we can integrate (\ref{vF}) to get the following: 
for any (not a.e.!) $x$ and $r$ with $x>r$,
\begin{eqnarray}
&\displaystyle 0\ge v'_{n_j}(x) - v'_{n_j}(r) + \int_r^x F_1[v'_{n_j}(s), \rho_{n_j}](s) \,ds
 \nonumber \\\nonumber \\
&\displaystyle  =\int_r^x \left(v''_{n_j}(x) + F_1[v'_{n_j}(s), \rho_{n_j}](s) \right) \,ds
 \nonumber \\\nonumber \\ \label{eF1}
&\displaystyle \ge 
\int_r^x ((v''_{n_j} - v''_{n_j+1})(s) - C|v'_{n_j} - v'_{n_j+1}|(s) - c(\rho_{n_j}-\rho_{n_j+1}))\,ds 
 \\\nonumber \\\nonumber 
&\displaystyle =v_{n_j}'(x) - v_{n_j}'(r) - v_{n_j+1}'(x) + v_{n_j+1}'(r) 
 \\\nonumber \\
&\displaystyle 
-  C\int_r^x |v'_{n_j} - v'_{n_j+1}|(s)ds 
- c (\rho_{n_j}-\rho_{n_j+1}) (x-r). \nonumber 
\end{eqnarray}
As it was explained earlier, due to the compactness in $C^1$ we may assume that 
$$
v_{n_j} \to \tilde v, \quad v'_{n_j} \to \tilde v', \quad \& \quad v'_{n_j+1} \to \tilde v' , \quad j\to\infty,
$$ 
in $C$ locally for some $\tilde v \in C^1$, as $j\to\infty$.  
Note that  $\tilde v'$ is absolutely continuous, which follows from the uniform local boundedness of $v_n''$. Therefore, 
it is possible to get to the limit in the inequality (\ref{eF1}) as $j \to  \infty$: for any $x>r$,
\[
\displaystyle 0\ge \tilde v'_{}(x) - \tilde v'(r) + \lim_{j\to\infty} \int_r^x F_1[s,v'_{n_j}(s), \rho_{n_j}] \,ds \ge  0, 
\]
since the right hand side in (\ref{eF1}) clearly goes to zero. 

Here
\begin{align*}
F_1[v'_{n_j}(s), \rho_{n_j}](s) = \inf_u \left[\frac{b^u}{a^u} v'_{n_j}(s) + \frac{f^u}{a^u}(s) - \frac{ \rho_{n_j}}{a^u}(s)\right] 
 \\\\
\to \inf_{u\in U}  \left[\frac{b^u}{a^u} \tilde v'(s) + \frac{f^u}{a^u}(s) - \frac{ \rho_n}{a^u}(s)\right]
= F_1[\tilde v'(s), \tilde\rho](s), \quad j \to \infty.  
\end{align*}
So, from (\ref{eF1}) we obtain the desired equation (\ref{133})
\[
\tilde v'(x) - \tilde v'(r) + \int_r^x F_1[s, \tilde v'(s), \tilde\rho]\,ds = 0. 
\]
In turn, since $F_1[\tilde v'(s), \tilde\rho](s)$ is continuous and absolutely continuous in $s$,  it implies $\tilde v \in C^2$, and by (well-defined) differentiation we get the equation~(\ref{1333}) for every $x\in \mathbb R$.

~

In the sequel it will follow from the uniqueness of solution to the Bellman's equation  that actually the whole sequence $v_n$ converges up to an additive constant sequence locally uniformly in $C^1$ to a single limit. However, it is not needed  for our proof. 

~

\noindent
{\bf 5}. {\em Uniqueness for $\rho$ in (\ref{ehjb}).} Assume that there are two solutions of the (HJB) equation, $(v^1, \rho^1)$ and $(v^2, \rho^2)$ with  $v^i\in {\cal K}$, $i=1,2$: 
\[
\inf\limits_{u\in U} (L^u v^1(x) + f^u(x) - \rho^1) = \inf\limits_{u\in U} (L^u v^2(x)+ f^u(x) - \rho^2) = 0.
\]
Earlier it was shown that both $v^1$ and $v^2$ are classical solutions with locally Lipschitz second derivatives. Let $w(x) := v^1(x) - v^2(x)$ and consider two strategies $\alpha_1, \alpha_2 \in {\cal A}$ such that 
$\alpha_1(x) \in \mbox{Argmax}_{u\in U} (L^u w(x))$ and $\alpha_2(x) \in  \mbox{Argmin}_{u\in U} (L^u w(x))$, and let $X^{1}_t, X^{2}_t$ be solutions of the SDEs corresponding to each strategy $\alpha_i$, $i=1,2$. Note that due to the measurable choice arguments -- see the Appendix -- such Borel strategies exist; corresponding weak solutions also exist. 
Let us denote
\[
h_1(x) :=\sup_{u\in U} (L^u w(x) - \rho^1 + \rho^2), \quad 
h_2(x) :=\inf_{u\in U} (L^u w(x) -  \rho^1 + \rho^2).
\]
Then, 
\begin{eqnarray*}
&h_2(x) 
 =  \inf_{u\in U} (L^u v^1(x) + f^u(x) - \rho^1 - (L^u v^2(x)+ f^u(x) - \rho^2))
 \\\\
&\le \inf\limits_{u\in U} (L^u v^1(x) + f^u(x) - \rho^1) - \inf\limits_{u\in U} (L^u v^2(x)+ f^u(x) - \rho^2) = 0.
\end{eqnarray*}
Similarly, 
\begin{eqnarray*}
&h_1(x) = - \inf_u (L^u (-v^2)(x) - \rho^2 + \rho^1) 
 \\\\
&=  -\inf_u (L^u v^2(x) + f^u(x) + \rho^2 - (L^u v^1(x)+ f^u(x) + \rho^1))
 \\\\
&\ge - \left[\inf_u (L^u v^2(x) + f^u(x) - \rho^2) - \inf_u (L^u v^1(x)+ f^u(x) - \rho^1)\right] = 0.
\end{eqnarray*}
We have,
\[
L^{\alpha_2} w(x)  = h_2(x) - \rho^2 + \rho^1,
\]
and
\[
L^{\alpha_1} w(x) = h_1(x)
- \rho^2 + \rho^1.
\]
Due to Dynkin's formula we have, 
\begin{eqnarray*}
\mathbb E_x w(X^1_t) - w(x) = \mathbb E_x\int_0^t L^{\alpha_1}w(X^1_s)\,ds 
 \\
= \mathbb E_x\int_0^t h_1(X^1_s)\,ds + (\rho^1-\rho^2)\,t \stackrel{(h_1\ge 0)}{\ge} (\rho^1-\rho^2)\,t.
\end{eqnarray*}
Since the left hand side here is bounded for  a fixed $x$, due to the Lemma \ref{lem2}  we get, 
\[
\rho^1 - \rho^2 \le 0.
\]
Similarly, considering $\alpha_2$ we conclude that
\begin{eqnarray*}
E_x w(X^2_t) - w(x) = E_x\int_0^t L^{\alpha_2}w(X^2_s)\,ds 
 \\\\
= E_x\int_0^t h_2(X^2_s)\,ds + (\rho^1-\rho^2)\,t.
\end{eqnarray*}
From here, due to the  boundedness of the left hand side (Lemma \ref{lem2}) we get,
\[
\rho^2 - \rho^1 = \liminf_{t\to 0}
(t^{-1}E_x\int_0^t h_2(X^2_s)\,ds) \stackrel{(h_2\le 0)}{\le} 0.
\]
Thus, $\rho^1 - \rho^2 \ge 0$ and, hence, 
\[
\rho^1=\rho^2.
\]

~

\noindent
{\bf 6}. {\em Why $\rho = \tilde \rho$?} Recall that for any initial $\alpha_0\in {\cal A}$, the sequence $\rho_n$ converges to the same value $\tilde \rho$, which is a unique component of solution of the equation (\ref{ehjb}). 
Let us take any $\epsilon >0$ and consider  a strategy $\alpha_0$ such that 
\[
\rho_0 = \rho^{\alpha_0} < \rho+ \epsilon.
\]
Since the sequence $(\rho_n)$ decreases, the limit $\tilde \rho$ must satisfy the same inequality, 
\[
\tilde \rho = \lim_{n\to\infty} \rho_n < \rho+ \epsilon.
\]
Due to uniqueness of $\tilde \rho$ as a component of solution of the equation (\ref{ehjb}) and since $\epsilon>0$ is arbitrary, we find that 
\[
\tilde \rho \le \rho.
\]
But also $\tilde \rho \ge \rho$ since $\tilde \rho$ is the infimum of the cost function values over a smaller -- just countable -- family of strategies. So, in fact, 
\[
\tilde\rho = \rho.
\]

~

\noindent
{\bf 7}. {\em Uniqueness for $V$.}  
Let us have another look at the earlier equations in the step 6, replacing $\rho^2-\rho^1$ by zero as we already know that the second component in the solution is unique:
\begin{eqnarray*}
\mathbb E_x w(X^1_t) - w(x)  = \mathbb E_x\int_0^t h_1(X^1_s)\,ds.
\end{eqnarray*}
Clearly, $h_1\ge 0$ with $h_1 \not =  0$ -- i.e.,  with $\Lambda(x: \, h_1(x) >0)> 0$ -- would imply that $\langle h_1, \mu_1\rangle \, > 0$, which contradicts a zero left hand side (after division by $t$ with $t\to\infty$). So, we conclude that
\[
h_1 = 0, \quad\mu_1-\mbox{a.s.}
\]
Since $\mu_1 \sim \Lambda$ due to (\ref{explicit}), by virtue of Krylov's estimate  we have  that $0\le \mathbb E_x\int_0^t h_1(X^1_s)\,ds \le N \| h_1\|_{L_{1}}= 0$. So, in fact, 
\begin{eqnarray}\label{ezero}
\mathbb E_x w(X^1_t) - w(x)  =  0. 
\end{eqnarray}

Further, from (\ref{ezero})
and due to the last statement of the Lemma \ref{lem2} it follows that 
\begin{eqnarray*}
w(x) = \lim_{t\to \infty} \mathbb E_x w(X^1_t) = \langle w, \mu_1\rangle. 
\end{eqnarray*}
Hence, $w(x)$ is a constant. Recall that uniqueness of the first component $V$ is stated up to a constant, and it was just established that  
\[
v^1(x) - v^2(x) = \mbox{const}.
\]

~

\noindent
{\bf 8.} 
Returning to the second statement of Theorem \ref{thm2}, note that due to uniqueness of the solution of the HJB equation, convergence of the whole sequence $(v_n)$ up to additive constants depending only on $n$ is to the unique limit $v$. 

~

\noindent
{\bf 9}. {\em Local Lipschitz for $\tilde v''$.} Recall that a certain additional regularity of the coefficients is assumed. We have from (\ref{1333}) and (\ref{vineq}), 
\[
|\tilde v''(x)| = |F_1[\tilde v'(x), \tilde\rho](x)| \le C (1+|\tilde v'(x)|) \le C(1+|x|).
\]
Therefore, it follows from the Cauchy Mean Value Theorem that 
\[
|\tilde v'(x) - \tilde v'(x')| \le  C(1+|x|^m+|x'|^m)|x-x'|. 
\]
So,  due to Lipschitz condition on $b^u,a^u$ in $x$ and in virtue of the nondegeneracy of $a^u$, 
\begin{align*} 
|\tilde v''(x) - \tilde v''(x')| = |F_1[\tilde v'(x), \tilde\rho](x) - F_1[\tilde v'(x), \tilde\rho](x')|
 \\\\
=  |\inf_u [\hat b^u(x) \tilde v'(x) + \hat f^u(x) -   \frac{\tilde\rho}{a^u(x)}] 
-  \inf_u [\hat b^u(x') \tilde v'(x') + \hat f^u(x') -   \frac{\tilde\rho}{a^u(x')}] |
 \\\\
\le  \sup_u  |\hat b^u(x) \tilde v'(x) + \hat f^u(x) -   \frac{\tilde\rho}{a^u(x)} 
-  \hat b^u(x')\tilde  v'(x') - \hat f^u(x') +   \frac{\tilde\rho}{a^u(x')}|
\\\\
\le C\left(|\tilde v'(x) - \tilde v'(x')| + |x-x'|\right) \le  C(1+|x|^m+|x'|^m)|x-x'|. 
\end{align*} 
The required local Lipschitz property of the function $\tilde v''$ has been verified.
 \end{proof}

\appendix
\section{On a measurable choice}
For the reader's convenience we repeat the main arguments from \cite{amv17} concerning the measurable choice a little bit more precisely. Recall that in the presentation of RIA in the beginning of the section 3 existence of a Borel measurable version of such a strategy was assumed,  which minimizes some function for any fixed $x$. 
In our case existence of such a Borel strategy can be justified by using Stschegolkow's
(Shchegolkov's) theorem \cite{Stschegolkow} (see also \cite[Satz 39]{Arsenin2}, or  \cite[Theorem 1]{BrownPurves}). According to this result, if any {\em section} of a (nonempty) Borel set $E$ in the direct product of two complete separable metric spaces is sigma-compact (i.e., equals a countable sum of closed bounded sets) then a Borel selection belonging to this set $E$ exists. 

In our case we have,  $F[v, \rho](x) = \inf_{u\in U} \left[L^{u}v(x)+f^{u}(x) -\rho\right]$. For a fixed $v$ representing any $v_n$ in the proof, denote $\chi(u,x):= L^{u}v(x)+f^{u}(x) -\rho$  
and $\bar \chi(x):=F[v, \rho](x)$, and let 
$E = \{(u,x): \chi(u,x) = \bar \chi(x)\}$. This set is nonempty because the minima here are attained for each $x$. 
Its section for any $x\in \mathbb R$ is 
$E_x := \{u: \chi(u,x) = \bar\chi(x)\}$. Any such section is nonempty and 
closed and, hence, Borel. Indeed, if $E_x \ni u_n \to u, \, n\to\infty$, then $\chi(u_n,x)\to \chi(u,x)$ due to the continuity of $\chi(\cdot,x)$. 

The set $E$ itself is Borel, too. To show this, take any $\epsilon>0$ and denote 
\[
E(\epsilon):= \{(u,x): \chi(u,x) - \bar \chi(x) < \epsilon\}.
\]
This set is Borel because the functions $\chi(u,x)$ and  $\bar \chi(x\}$ are: the latter one since the minimum in $\min_u\chi(u,x)$ can be taken over some countable dense subset of $U$. 
(Recall that the second derivative $v''$ is Borel measurable by our convention.) It remains to note that 
\[
E = \bigcap_{k=1}^\infty E(1/k), 
\]
so that $E$ is also Borel. 

Thus, Stschegolkow's theorem is applicable  and, hence, a Borel measurable improved strategy $\alpha_{n+1}$ in the induction step of the RIA does exist for each step $n$. By the same reason Borel strategies $\alpha_1$ and $\alpha_2$ exist in the steps 6 and (implicitly) 8.

\end{document}